\documentclass[12pt]{article}
\usepackage{amsmath,amssymb,amsthm}
\newtheorem{thm}{Theorem}[section]

 \newtheorem{cor}{Corollary}[section]
 
 \newtheorem{prop}{Proposition}[section]
 \newtheorem{defn}{Definition}[section]
\newtheorem{rem}{Remark}[section]

\begin{document}
\begin{center}
{\large{\bf  $L^p$-$L^{q}$-$L^{r}$ estimates and minimal decay regularity for compressible Euler-Maxwell equations}}
\end{center}
\begin{center}
\footnotesize{Jiang Xu}\\[2ex]
\footnotesize{Department of Mathematics, \\ Nanjing
University of Aeronautics and Astronautics, \\
Nanjing 211106, P.R.China,}\\
\footnotesize{jiangxu\underline{ }79@nuaa.edu.cn}\\

\vspace{3mm}

\footnotesize{Faculty of Mathematics, \\ Kyushu University, Fukuoka 819-0395, Japan}\\
\vspace{10mm}

\footnotesize{Naofumi Mori}\\[2ex]
\footnotesize{Graduate School of Mathematics,\\ Kyushu University, Fukuoka 819-0395, Japan,}\\
\footnotesize{n-mori@math.kyushu-u.ac.jp}\\

\vspace{10mm}

\footnotesize{Shuichi Kawashima}\\[2ex]
\footnotesize{Faculty of Mathematics, \\ Kyushu University, Fukuoka 819-0395, Japan,}\\
\footnotesize{kawashim@math.kyushu-u.ac.jp}\\
\end{center}
\vspace{6mm}

\begin{abstract}
Due to the dissipative structure of \textit{regularity-loss}, extra higher regularity than that for the global-in-time existence
is usually imposed to obtain the optimal decay rates of classical solutions to dissipative systems. The aim of this paper is to seek
the lowest regularity index for the optimal decay rate of $L^{1}(\mathbb{R}^n)$-$L^2(\mathbb{R}^n)$. Consequently,
a notion of minimal decay regularity for dissipative systems of regularity-loss is firstly proposed. To do this, we develop a new time-decay estimate of $L^p(\mathbb{R}^n)$-$L^{q}(\mathbb{R}^n)$-$L^{r}(\mathbb{R}^n)$ type by using the low-frequency and high-frequency analysis in Fourier spaces.
As an application, for compressible Euler-Maxwell equations with the weaker dissipative mechanism, it is shown that
the minimal decay regularity coincides with the critical regularity for global classical solutions.
Moreover, the recent decay property for symmetric hyperbolic systems with non-symmetric dissipation is also extended to be the $L^p$-version.
\end{abstract}


\section{Introduction}\setcounter{equation}{0}
Compressible Euler-Maxwell equations appear in the mathematical modelling of semiconductor sciences. When semiconductor devices are operated under some high frequency conditions, such as photoconductive switches, electro-optics, semiconductor lasers and high-speed computers, \textit{etc.}, the electron transport in devices interacts with the propagating electromagnetic waves. Consequently, the Euler-Maxwell equations which take the the form
of Euler equations for the conservation laws of mass density, current density and energy
density for electrons, coupled to Maxwell's equations for self-consistent electromagnetic
field, are introduced to describe the transport process, the reader is also referred to \cite{CJW,MRS} for more explanation.

Let us consider the isentropic Euler-Maxwell equations where the energy equation is replaced with state equation of the pressure-density
relation. Precisely,
\begin{equation}
\left\{
\begin{array}{l}\partial_{t}n+\nabla\cdot(nu)=0,\\
 \partial_{t}(nu)+\nabla\cdot(nu\otimes u)+\nabla p(n)=-n(E+u\times
B)-nu,\\
\partial_{t}E-\nabla\times B=
nu,\\
\partial_{t}B+\nabla\times E=0,\\
 \end{array} \right.\label{R-E1}
\end{equation}
with constraints
\begin{equation}
\nabla\cdot E=n_{\infty}-n,\ \ \ \nabla\cdot B=0 \label{R-E2}
\end{equation}
for $(t,x)\in[0,+\infty)\times\mathbb{R}^{3}$.
Here the unknowns
$n>0, u\in\mathbb{R}^{3}$ are the density and the velocity of electrons,
and $E\in \mathbb{R}^{3}, B\in \mathbb{R}^{3}$
denote the electric field and magnetic field, respectively. The pressure $p(n)$ is a
given smooth function of $n$ satisfying $p'(n)>0$ for $n>0$. $n_{\infty}$ is assumed to be
a positive constant, which
stands for the density of positively charged background ions.
Observe that the system (\ref{R-E1}) admits a constant equilibrium state
\begin{equation}
(n_{\infty},0,0,B_{\infty}), \label{R-E3}
\end{equation}
which are regarded as vectors in $\mathbb{R}^{10}$, where $B_{\infty}\in \mathbb{R}^{3}$ is an arbitrary fixed constant vector.
In this paper, we are concerned with (\ref{R-E1})-(\ref{R-E2}) with the initial data
\begin{equation}
(n,u,E,B)|_{t=0}=(n_{0},u_{0},E_{0},B_{0})(x),\ \ x\in \mathbb{R}^{3}. \label{R-E4}
\end{equation}
We see that (1.2) can hold for any $t>0$ if the initial data satisfy
\begin{equation}
\nabla\cdot E_{0}=n_{\infty}-n_{0},\ \ \
\nabla\cdot B_{0}=0, \ \ x\in\mathbb{R}^{3}. \label{R-E5}
\end{equation}

System (\ref{R-E1}) is partially dissipative due to the damping term in the momentum equations.
In one dimensional space, by using the Godunov scheme with the fractional step together with the compensated compactness
theory, Chen, Jerome, and Wang \cite{CJW} constructed the global existence of weak solutions. In
the multidimensional space, the question of global weak solution of (\ref{R-E1}) is quite open, and only the global existence and large-time behavior of smooth solutions have been studied. To state these known results,
it is convenient to reformulate the system (\ref{R-E1}) as
\begin{equation}
\left\{
\begin{array}{l}\partial_{t}n+u\cdot\nabla n+n\mathrm{div}u=0,\\
 \partial_{t}u+(n\cdot\nabla)u+a(n)\nabla n+E+u\times
B+u=0,\\
\partial_{t}E-\nabla\times B-
nu=0,\\
\partial_{t}B+\nabla\times E=0,\\
 \end{array} \right.\label{R-E6}
\end{equation}
where $a(n):=p'(n)/n$ is the enthalpy function. For simplicity, we set $w=(n,u,E,B)^{\top}$ ($\top$ transpose), which is a column vector in $\mathbb{R}^{10}$.
Then (\ref{R-E6})  can be written in the vector form
\begin{equation}
A^{0}(w)w_{t}+\sum_{j=1}^{3}A^{j}(w)w_{x_{j}}+L(w)w=0,\label{R-E7}
\end{equation}
where the coefficient matrices are given explicitly as
\begin{eqnarray*}
A^{0}(w)=\left(
           \begin{array}{cccc}
             a(n) & 0 & 0 & 0 \\
             0 & nI & 0 & 0 \\
             0 & 0 & I & 0 \\
             0 & 0 & 0 & I \\
           \end{array}
         \right),\ L(w)=\left(
       \begin{array}{cccc}
         0 & 0 & 0 & 0 \\
         0 & n(I-\Omega_{B}) & nI & 0 \\
         0 & -nI & 0 & 0 \\
         0 & 0 & 0 & 0 \\
       \end{array}
     \right),
     \end{eqnarray*}
\begin{eqnarray*}
\sum_{j=1}^{3}A^{j}(w)\xi_{j}=\left(
                                \begin{array}{cccc}
                                  a(n)(u\cdot\xi) & p'(n)\xi & 0 & 0 \\
                                  p'(n)\xi^{\top} & n(u\cdot\xi)I & 0 & 0 \\
                                  0 & 0 & 0 & -\Omega_{\xi} \\
                                  0 & 0 & \Omega_{\xi} & 0 \\
                                \end{array}
                              \right).
\end{eqnarray*}
Here $I$ is the identity matrix of third order, $\xi=(\xi_{1},\xi_{2},\xi_{3})\in \mathbb{R}^3$, and $\Omega_{\xi}$ is the skew-symmetric matrix defined by
\begin{eqnarray*}
\Omega_{\xi}=\left(
               \begin{array}{ccc}
                 0 & -\xi_{3} & \xi_{2} \\
                 \xi_{3} & 0 & -\xi_{1} \\
                 -\xi_{2} & \xi_{1} & 0 \\
               \end{array}
             \right)
\end{eqnarray*}
such that $\Omega_{\xi} E^{\top}=(\xi\times E)^{\top}$ (as a column vector in $\mathbb{R}^{3}$) for $E=(E_{1},E_{2},E_{3})\in \mathbb{R}^{3}$.
Let us mention that (\ref{R-E7}) is a symmetric hyperbolic system, since
$A^{0}(w)$ is real symmetric and positive definite and $A^{j}(w)(j=1,2,3)$ are real symmetric.  It is easy to check that the dissipative matrix $L(w)$ is nonnegative definite, however, $L(w)$ \textit{is not real symmetric}. Such partial dissipation forces (\ref{R-E1}) to go beyond
the class of generally dissipative hyperbolic systems satisfying the Kawashima-Shizuta condition, which have been well
studied by \cite{BHN,HN,Ka1,KY,KY2,SK,UKS,Y} in spatially Sobolev spaces, \cite{XK1,XK2} in critical Besov spaces and therein references.

So far there are also a number of efforts on global smooth solutions for (\ref{R-E1}) by various authors, see \cite{D1,DLZ,P,PWG,TWW,UWK}
on the framework of Sobolev spaces, \cite{X,XXK} in critical Besov spaces. In the present paper, we pay attention to
global smooth solutions constructed in Sobolev spaces.
Here we adopt those notations formulated by \cite{UWK}:
$$w_{\infty}=(n_{\infty},0,0,B_{\infty})^{\top}, \ \ \ w_{0}=(n_{0},u_{0},E_{0},B_{0})^{\top};$$
$$N_{0}(t):=\sup_{0\leq \tau \leq t}\|(w-w_{\infty})(\tau)\|_{H^{s}};$$
$$D_{0}(t)^2:=\int^{t}_{0}(\|(n-n_{\infty},u)(\tau)\|^2_{H^{s}}+\|E(\tau)\|^2_{H^{s-1}}+\|\nabla B(\tau)\|^2_{H^{s-2}})d\tau.$$
The global existence of smooth solutions is drawn briefly as follows:
\begin{equation}
\left\{
\begin{array}{l} \mbox{The initial data} \
w_{0}-w_{\infty}\in H^{s}(s\geq3) \ \mbox{and} \ (1.5),\\
I_{0}:=\|w_{0}-w_{\infty}\|_{H^{s}}\leq \varepsilon_{0},\\
 \end{array} \right.\label{R-E8}
\end{equation}
$\Rightarrow$
\begin{equation}
\left\{
\begin{array}{l}
w-w_{\infty}\in C([0,\infty);H^{s})\cap C^{1}([0,\infty);H^{s-1}),\\
N_{0}(t)^2+D_{0}(t)^2\leq CI_{0}^2.
 \end{array} \right.\label{R-E9}
\end{equation}

\begin{rem}\label{rem1.1}
Due to the non-symmetric dissipation, there is an $1$-regularity-loss phenomenon of dissipation rates from the electromagnetic part $(E,B)$, which is first observed by Duan \cite{D1}. However, the regularity index $s\geq4$ was needed. Subsequently, Ueda, Wang and the third author \cite{UWK} improved his result by the usual energy method such that the regularity was relaxed as $s\geq3$. Denote by $s_{c}$ the critical regularity for global classical solutions.
It follows from the basic local-in-time theory of Kato and Majda in \cite{K,M} that ``$s_{c}=3$" for the Euler-Maxwell system in $\mathbb{R}^{3}$.
\end{rem}

We would like to note that all physical parameters are normalized to be one in (\ref{R-E1}). If considered, there are rigorous justifications on the singular-parameter limits for (\ref{R-E1}). For instance, Peng and Wang \cite{PW1,PW2,PW3}
justified the nonrelativistic limit, quasi-neutral limit, and the combined nonrelativistic and
quasi-neutral limits of (\ref{R-E1}). Inspired by Maxwell-type iteration and a continuation principle developed in \cite{Y3}, the first author \cite{XX} constructed new approximations and proved different diffusive limits from (\ref{R-E1}) to drift-diffusion models and energy-transport models,
and the corresponding convergence order was also obtained.

In this paper, we focus on the time decay of classical solutions to (\ref{R-E1}), which is an interesting question.
For this purpose, we write (\ref{R-E1}) as
the linearized perturbation form around the equilibrium state $w_{\infty}$.  Set $\upsilon=nu/n_{\infty}$. Then
\begin{equation}
\left\{
\begin{array}{l}\partial_{t}n+n_{\infty}\mathrm{div}\upsilon=0,\\
\partial_{t}\upsilon+a_{\infty}\nabla n+E+\upsilon\times B+\upsilon=(\mathrm{div}q_{2}+r_{2})/n_{\infty},
\\
\partial_{t}E-\nabla\times B-n_{\infty}\upsilon=0,\\
\partial_{t}B+\nabla\times E=0,\\
 \end{array} \right.\label{R-E10}
\end{equation}
where $a_{\infty}=p'(n_{\infty})/n_{\infty}$,
$$q_{2}=-n_{\infty}^2\upsilon\otimes \upsilon/n-[p(n)-p(n_{\infty})-p'(n_{\infty})(n-n_{\infty})]I $$
and
$$r_{2}=-(n-n_{\infty})E-n_{\infty}\upsilon\times(B-B_{\infty}).$$
We put $z:=(\rho, \upsilon, E, h)^{\top}$, where $\rho=n-n_{\infty}$ and $h=B-B_{\infty}$. The corresponding initial data are given by
\begin{equation}
z|_{t=0}=(\rho_{0}, \upsilon_{0}, E_{0}, h_{0})^{\top}(x) \label{R-E11}
\end{equation}
with $\rho_{0}=n_{0}-n_{\infty},\ \upsilon_{0}=n_{0}u_{0}/n_{\infty}$ and $h_{0}=B_{0}-B_{\infty}$.
System (\ref{R-E10}) is also rewrite in the vector form as
\begin{equation}
A^{0}z_{t}+\sum_{j=1}^{3}A^{j}z_{x_{j}}+Lz=\sum_{j=1}^{3}Q_{x_{j}}+R, \label{R-E12}
\end{equation}
where $A^{0}, A^{j} $ and $L$ are the constant matrices given by (\ref{R-E7}) with $w=w_{\infty}$, $Q(z)=(0, q^{j}_{2}/n_{\infty}, 0, 0)^{\top}$ and $R(z)=(0,r_{2}/n_{\infty}, 0,0)^{\top}$. Noticing that $Q(z)=O(|(\rho,\upsilon)|^2)$ and $R(z)=O(\rho |E|+|\upsilon||h|)$,
which is a nice degenerate structure for our subsequent decay analysis for the Euler-Maxwell system (\ref{R-E1}).

The linearized form of (\ref{R-E12}) reads as
\begin{equation}
A^{0}\partial_{t}z_{\mathcal{L}}+\sum_{j=1}^{3}A^{j}\partial_{x_{j}}z_{\mathcal{L}}+Lz_{\mathcal{L}}=0, \label{R-E13}
\end{equation}
and the corresponding initial data $z_{0}:=(\rho_{0}, \upsilon_{0}, E_{0}, h_{0})^{\top}$ satisfy
\begin{equation}
\mathrm{div}E_{0}=-\rho_{0},\ \ \ \mathrm{div}h_{0}=0. \label{R-E133}
\end{equation}
Next, we apply the Fourier transform of (\ref{R-E13}) to get
\begin{equation}
A^{0}\partial_{t}\widehat{z_{\mathcal{L}}}+i|\xi|A(\omega)\widehat{z_{\mathcal{L}}}+L\widehat{z_{\mathcal{L}}}=0, \label{R-E14}
\end{equation}
where $A(\omega)=\sum_{j=1}^{3}A^{j}\omega_{j}$, and $\omega=\xi/|\xi|\in \mathbb{S}^2$. Set
$$\hat{\Phi}(\xi)=(A^{0})^{-1}(i|\xi|A(\omega)+L).$$
We define the Green matrix $\widehat{\mathcal{G}(t)z}:=e^{-t\hat{\Phi}(\xi)}\hat{z}$
which is a mapping from $X_{\xi}$ to $X_{\xi}$ with $X_{\xi}=\{\hat{z}\in\mathbb{C}^{10}: i\xi\cdot\hat{E}=-\hat{\rho},\ i\xi\cdot\hat{h}=0 \}\subset \mathbb{C}^{10}$.
Then the linearized solution $z_{\mathcal{L}}$ of (\ref{R-E13})-(\ref{R-E133}) is given by $\mathcal{G}(t)z_{0}$.

As shown by \cite{UK}, by using the energy method in Fourier spaces,
the Fourier image of $z_{\mathcal{L}}$ satisfies the following pointwise estimate
\begin{eqnarray}
|\widehat{z_{\mathcal{L}}}(t,\xi)|\lesssim e^{-c_{0}\eta(\xi)t} |\hat{z}_{0}|\label{R-E15}
\end{eqnarray}
for any $t\geq0$ and $\xi\in \mathbb{R}^{3}$, where the dissipative rate $\eta(\xi)=|\xi|^2/(1+|\xi|^2)^2$ and $c_{0}>0$ is a constant.
Furthermore, the decay property was achieved:
\begin{eqnarray}
\|\partial^{k}_{x}z_{\mathcal{L}}\|_{L^2}\lesssim (1+t)^{-3/4-k/2}\|z_{0}\|_{L^1}+(1+t)^{-\ell/2}\|\partial^{k+\ell}_{x}z_{0}\|_{L^2}, \label{R-E16}
\end{eqnarray}
where $k$ and $\ell$ are non-negative integers.

\begin{rem}\label{rem1.2}
The decay (\ref{R-E16}) is of the regularity-loss type, since $(1+t)^{-\ell/2}$ is created by assuming the additional $\ell$-th order regularity on the initial data. Consequently, for the nonlinear Euler-Maxwell system, extra higher regularity than that for global-in-time existence of classical solutions
is imposed to obtain the optimal decay rates. Actually, the similar phenomena also appear in the study of other dissipative systems, such as
quasi-linear hyperbolic systems of viscoelasticity in \cite{D,DNK}, hyperbolic-elliptic systems of radiating gas in \cite{HK},
dissipative Timoshenko systems in \cite{IK,LK}, Vlasov-Maxwell-Boltzmann system in \cite{DS}, and a plate equation with rotational inertia effect in \cite{SK2}, \textit{etc.}.
\end{rem}

A natural question follows. Which index characterises the minimal regularity for the optimal time decay for dissipative systems of regularity-loss? This motivates the following
\begin{defn}\label{defn1.1}
If the optimal decay rate of $L^{1}(\mathbb{R}^n)$-$L^2(\mathbb{R}^n)$ type is achieved under the lowest regularity assumption, then the lowest index is called the minimal decay regularity index of dissipative systems of regularity-loss, which is labelled as $s_{D}$.
\end{defn}
According to dissipative systems of regularity loss mentioned in Remark \ref{rem1.2}, it is not difficult to see that $s_{D}>s_{c}$.
For instance, like compressible Euler-Maxwell equations (\ref{R-E1}), we see that $s_{c}=3$, whereas $s_{D}=6$
was shown by Duan and his collaborators \cite{D1,DLZ}, Ueda and the third author \cite{UK} independently.
If the regularity of initial data is imposed higher than $s_{D}=6$, then
more decay information of solutions can be available, \textit{e.g.}, see \cite{D1,DLZ,TWW}. However, this is beyond our primary interest.

The interest of this paper is to seek the minimal decay regularity for (\ref{R-E1}), which improves previous efforts such that $s_{D}=3$, which is exactly the critical regularity $s_{c}=3$ for global classical solutions. Clearly, the regularity assumption on the initial data is reduced heavily in comparison with known results in \cite{D1,DLZ,TWW,UK} and so on.
Due to the less regularity assumption, previous techniques used are invalidated. In what follows, let us explain new technical points and
the strategy to overcome main difficulties.
Firstly, our improvement lies in a new decay inequality of $L^p$-$L^q$-$L^r$ type, which is the most crucial ingredient.
To the best of our knowledge, Umeda, the third author and Shizuta initialled a decay inequality of $L^2$-$L^{q}$-$L^2$ type with $\eta(\xi)=|\xi|^2/(1+|\xi|^2)$ in the earlier work \cite{UKS} for hyperbolic-parabolic systems satisfying the Kawashima-Shizuta condition, where the high frequency part yields an exponential decay (\textit{without regularity-loss}). Subsequently, the third author and his collaborators
investigated decay properties for dissipative systems of regularity-loss, as in \cite{D,DNK,HK,IHK,IK,MK,SK2,UDK,UK}. The corresponding dissipative rate is subjected to the $(a,b)$-type: $\eta(\xi)=|\xi|^{2a}/(1+|\xi|^2)^{b}$, where $(a,b)$ is a pair of positive integers. The high-frequency part usually admits a polynomial decay if the initial data is imposed more regularity, for example, see (\ref{R-E16}). However, so far these known decay properties are all restricted to be the $L^2$-$L^{q}$-$L^2$ estimates, even if in the mixed space containing the microscopic velocity,  see the Vlasov-Maxwell-Boltzmann system in \cite{DS}.
In this paper,  we first present the following general $L^p$-$L^{q}$-$L^{r}$ time decay estimate.
\begin{thm}\label{thm1.1}  ($L^p$-$L^q$-$L^r$ estimates)
Let $\eta(\xi)$ be a positive, continuous and real-valued function in $\mathbb{R}^{n}$ satisfying
\begin{eqnarray}\eta(\xi)\sim\left\{
                 \begin{array}{ll}
                   |\xi|^{\sigma_{1}}, &  |\xi|\rightarrow 0; \\
                   |\xi|^{-\sigma_{2}}, & |\xi|\rightarrow\infty;
                 \end{array}
               \right. \label{R-E166}
\end{eqnarray}
for $\sigma_{1}, \sigma_{2}>0$.  For $\phi\in \mathcal{S}(\mathbb{R}^{n})$, it holds that
\begin{eqnarray}
&&\|\mathcal{F}^{-1}[|\xi|^{k}e^{-\eta(\xi)t}|\hat{\phi}(\xi)|]\|_{L^p}
\nonumber\\ &\leq& C\underbrace{(1+t)^{-\gamma_{\sigma_{1}}(q,p)-\frac{k-j}{\sigma_{1}}}\|\partial_{x}^{j}\phi\|_{L^{q}}}_{Low-frequency\  Estimate}
+\underbrace{(1+t)^{-\frac{\ell}{\sigma_{2}}+\gamma_{\sigma_{2}}(r,p)}\|\partial_{x}^{k+\ell}\phi\|_{L^r}}_{High-frequency\  Estimate}, \label{R-E167}
\end{eqnarray}
for $\ell>n(\frac{1}{r}-\frac{1}{p})$\ \footnote{We would like to remark that $\ell\geq0$ in the case of $p=r=2$, which is the same assumption as in Corollary \ref{cor1.1}.},\ $1\leq q,r\leq2\leq p\leq\infty$ and $ 0\leq j\leq k,$ where $\gamma_\sigma(q,p):=\frac{n}{\sigma}(\frac{1}{q}-\frac{1}{p})(\sigma>0)$
and $C$ is some positive constant.
\end{thm}
\begin{rem} \label{rem1.3}
The proof of Theorem \ref{thm1.1} depends on the traditional low-frequency and high-frequency decomposition methods, see Sect.\ref{sec:2}.
For the low-frequency part, the function decays like a generalized heat kernel. For the high-frequency part, it decays in time not only with algebraic rates of any order as long as the function is spatially regular enough, but also additional information related the integrability is captured in comparison with (\ref{R-E16}). We realized that the inequality (\ref{R-E167}) would be have the great potential for applications in the study of dissipative systems with regularity-loss.
\end{rem}

In the present paper, for the convenience of application, we also give a $L^p$-$L^q$-$L^r$ estimate for the dissipative rate of $(1,2)$-type.
Obviously, $\sigma_{1}=\sigma_{2}=2$ in this case.
\begin{cor}\label{cor1.1}
Let $\eta(\xi)=|\xi|^2/(1+|\xi|^2)^2$. For $\phi\in \mathcal{S}(\mathbb{R}^{n})$, it holds that
\begin{eqnarray}
&&\|\mathcal{F}^{-1}[|\xi|^{k}e^{-\eta(\xi)t}|\hat{\phi}(\xi)|]\|_{L^p}
\nonumber\\ &\leq& C\underbrace{(1+t)^{-\gamma_{2}(q,p)-\frac{k-j}{2}}\|\partial_{x}^{j}\phi\|_{L^{q}}}_{Low-frequency\  Estimate}
+\underbrace{(1+t)^{-\frac{\ell}{2}+\gamma_{2}(r,p)}\|\partial_{x}^{k+\ell}\phi\|_{L^r}}_{High-frequency\  Estimate}, \label{R-E168}
\end{eqnarray}
for $\ell>n(\frac{1}{r}-\frac{1}{p})$,\ $1\leq q,r\leq2\leq p\leq\infty$ and $ 0\leq j\leq k,$ where $C>0$ is some constant.
\end{cor}

By virtue of the new decay estimate (\ref{R-E168}), we focus on the minimal decay regularity for the Euler-Maxwell equations (\ref{R-E1}), since it has the dissipative structure of $(1,2)$-type. Because of less regularity,  main time-weighted estimates related to
the  norm
$$W^{\bot}(t):=\sup_{0\leq \tau \leq t}(1+\tau)\|(\rho,\upsilon,E)\|_{W^{1,\infty}}$$
in \cite{UK} no longer works in the semigroup approach.
To overcome the first technical obstruction, we have to skip the semigroup approach.
Consequently, the energy method in Fourier spaces for the nonlinear system (\ref{R-E10})-(\ref{R-E11})
is mainly performed. The second obstruction comes from the high-frequency estimate of nonlinear terms.
Fortunately, the inequality (\ref{R-E168}) plays an important role to get round the difficulty. More precisely,
the high-frequency estimate is divided into two parts, and on each part, the advantage of (\ref{R-E168}) than (\ref{R-E16})
is that different values (for example, $r=1$ or $r=2)$ can be chosen to obtain desired decay estimates. See (\ref{R-E39}), (\ref{R-E42}) and
(\ref{R-E44})-(\ref{R-E45}) for more details. Additionally, it should be pointed out that
there is a new observation on the degenerate structure of nonlinear terms in (\ref{R-E10}) or (\ref{R-E12}),
which is also helpful to obtain the optimal decay. Our decay result of (\ref{R-E7}) is stated as follows.
\begin{thm}\label{thm1.2}
Assume that the initial data satisfy $w-w_{\infty}\in H^{3}\cap L^1$ and (\ref{R-E5}). Set $I_{1}:=\|w-w_{\infty}\|_{H^{3}\cap L^1}.$
Then there exists a positive $\varepsilon_{1}$ such that if $I_{1}\leq\varepsilon_{1}$, then the classical solution of
the Cauchy problem of (\ref{R-E7}) admits the optimal decay estimate
\begin{eqnarray}
\|w-w_{\infty}\|_{L^2}\leq C \|w_{0}-w_{\infty}\|_{H^{3}\cap L^1}(1+t)^{-3/4}, \label{R-E17}
\end{eqnarray}
where $C>0$ is some constant.
\end{thm}

\begin{rem}\label{rem1.4}
From Theorem \ref{thm1.2}, we can see that $s_{D}=3$, which coincides with
$s_{c}=3$ for global classical solutions of (\ref{R-E7}). In other words, the extra regularity is not necessary, which reduces the regularity assumption heavily in comparison with previous efforts mentioned in Remark \ref{rem1.2}. It is worth noting that this is the first result for dissipative
systems of regularity-loss type. Therefore, two research lines will begin in the near future. One is to 
employ $L^p$-$L^q$-$L^r$ estimates in Theorem \ref{thm1.1} for more dissipative rates of $(a,b)$ type, and then
investigate dissipative systems of regularity-loss on the level of continuum mechanics. Another is to develop 
$L^p$-$L^q$-$L^r$ estimates on the kinetic level, and then study the Vlasov-Maxwell-Boltzmann system of regularity-loss. 
\end{rem}

\begin{rem}\label{rem1.5}
From the point of view of harmonic analysis, the first and third authors investigated (\ref{R-E1}) in the spatially critical Besov spaces, where the regularity index $s_{c}=5/2$ (see \cite{X,XXK}). Very recently,  the first and third authors also gave a new decay
framework $L^2\cap\dot{B}^{-s}_{2,\infty},\ 0<s\leq n/2\ (n\geq1)$ for general dissipative system satisfying the Kawashima-Shizuta condition (see \cite{XK2}). Based on the decay framework, is there some improved room such that $s_{D}=5/2$ for Euler-Maxwell equations?  Here, we draw down an open question.
\end{rem}

\textbf{Notations}. Throughout the paper, use $\langle\cdot,\cdot \rangle$ to denote the standard inner product in the
complex vector value $\mathbb{C}^{n}\ (n\geq1)$. $f\lesssim g$ means $f\leq Cg$, where $C>0$
is a generic constant. $f\thickapprox g$ means $f\lesssim g$ and $g\lesssim f$ simultaneously. The Fourier transform $\hat{f}$ (or $\mathcal{F}[f]$)
of a function $f\in \mathcal{S}$ (the Schwarz class)
is denoted by $$\mathcal{F}[f]:=\int_{\mathbb{R}^{n}}f(x)e^{-2\pi x\cdot\xi}dx.$$
The Fourier transform of a tempered function in $\mathcal{S}'$ is defined by
the dual argument in the standard way. $\mathcal{F}^{-1}[f]$ stands for the inverse Fourier transform in $\mathbb{R}^{n}$.

The rest of this paper unfolds as follows. In Sect.\ref{sec:2}, we shall prove
the $L^p$-$L^q$-$L^r$ decay estimates. Sect.\ref{sec:3} is devoted to
develop the energy method in Fourier spaces for (\ref{R-E10}) with nonlinear
terms. Furthermore, with aid of $L^p$-$L^q$-$L^r$ estimates, the optimal decay rate for (\ref{R-E10}) under
the minimal regularity assumption is shown. In Appendix (Sect.\ref{sec:4}),
 we present a general $L^p$-version of the decay property  for symmetric hyperbolic systems with non-symmetric dissipation.

\section{The proof of $L^p$-$L^q$-$L^r$ estimates}\label{sec:2}
In this section, we first give the proof for the time-decay estimate of $L^p$-$L^q$-$L^r$ type by using the high-frequency and low-frequency decomposition method.

\begin{proof}

[\textbf{The proof of Theorem \ref{thm1.1}}]

Firstly, it follows from Hausdorff-Young's inequality that
\begin{eqnarray}
\|\mathcal{F}^{-1}[|\xi|^{k}e^{-\eta(\xi)t}|\hat{\phi}(\xi)|]\|_{L^p}\lesssim \||\xi|^{k}e^{-\eta(\xi)t}|\hat{\phi}(\xi)|\|_{L^{p'}},\label{R-E299}
\end{eqnarray}
for $ 1/p+1/p'=1, 2\leq p\leq\infty. $

Secondly, we deal with the $L^{p'}$-norm on the right-hand side of (\ref{R-E299}) at the low-frequency and high-frequency, respectively. It follows from the assumption (\ref{R-E166}) that there exists a constant $R_{0}>0$ such that
\begin{eqnarray}
&&\||\xi|^{k}e^{-\eta(\xi)t}|\hat{\phi}(\xi)|\|_{L^{p'}}\nonumber\\
&\leq& \||\xi|^{k}e^{-c|\xi|^{\sigma_{1}}t}|\hat{\phi}(\xi)|\|_{L^{p'}(|\xi|\leq R_{0})}+
\||\xi|^{k}e^{-c|\xi|^{-\sigma_{2}}t}|\hat{\phi}(\xi)|\|_{L^{p'}(|\xi|\geq R_{0})}
\nonumber\\
&\triangleq& I_{1}+I_{2}, \label{R-E30}
\end{eqnarray}
for some constant $c>0$.

For $I_{1}$,  we are led to the estimate
\begin{eqnarray}
I_{1}&\leq& \||\xi|^{k}e^{-c|\xi|^{\sigma_{1}}t}|\hat{\phi}(\xi)|\|_{L^{p'}(|\xi|\leq R_{0})}
\nonumber\\
&=& \||\xi|^{j}|\hat{\phi}(\xi)| |\xi|^{(k-j)}e^{-c|\xi|^{\sigma_{1}}t}\|_{L^{p'}(|\xi|\leq R_{0})}
 \ (0\leq j\leq k)\nonumber\\
&\leq& \||\xi|^{j}\hat{\phi}\|_{L^{q'}(|\xi|\leq R_{0})}\||\xi|^{(k-j)}e^{-c|\xi|^{\sigma_{1}}t}\|_{L^{s_{1}}(|\xi|\leq R_{0})}
 \ \ \Big(\frac{1}{q'}+\frac{1}{s_{1}}=\frac{1}{p'},\ q'\geq2\Big)
\nonumber\\
&\lesssim&  \|\partial_{x}^{j}\phi\|_{L^{q}} (1+t)^{-\frac{n}{\sigma_{1}s_{1}}-\frac{k-j}{\sigma_{1}}} \ \ \Big(\frac{1}{q}+\frac{1}{q'}=1\Big)\nonumber\\
&\lesssim& (1+t)^{-\frac{n}{\sigma_{1}}(\frac{1}{q}-\frac{1}{p})-\frac{k-j}{\sigma_{1}}}\|\partial_{x}^{j}\phi\|_{L^{q}}, \label{R-E31}
\end{eqnarray}
where H\"{o}lder inequality was used in the third line and Hausdorff-Young's inequality was used again in the fourth line.

For $I_{2}$, we arrive at
\begin{eqnarray}
I_{2}&\leq& \||\xi|^{k}e^{-ct/|\xi|^{\sigma_{2}}}|\hat{\phi}(\xi)|\|_{L^{p'}(|\xi|\geq R_{0})}
\nonumber\\ &=&\Big\||\xi|^{k+\ell}|\hat{\phi}(\xi)|\frac{e^{-ct/|\xi|^{\sigma_{2}}}}{|\xi|^{\ell}}\Big\|_{L^{p'}(|\xi|\geq R_{0})}
\nonumber\\&\leq&
\||\xi|^{k+\ell}\hat{\phi}\|_{L^{r'}(|\xi|\geq R_{0})}\Big\|\frac{e^{-ct/|\xi|^{\sigma_{2}}}}{|\xi|^{\ell}}\Big\|_{L^{s_{2}}(|\xi|\geq R_{0})}
 \ \ \Big(\frac{1}{r'}+\frac{1}{s_{2}}=\frac{1}{p'},\ r'\geq2 \Big)
\nonumber\\&\lesssim& \|\partial_{x}^{k+\ell}\phi\|_{L^{r}}\Big\|\frac{e^{-ct/|\xi|^{\sigma_{2}}}}{|\xi|^{\ell}}\Big\|_{L^{s_{2}}(|\xi|\geq R_{0})}
\ \  \Big(\frac{1}{r}+\frac{1}{r'}=1\Big), \label{R-E32}
\end{eqnarray}
where
\begin{eqnarray}\int_{|\xi|\geq R_{0}}\frac{e^{-ct/|\xi|^{\sigma_{2}}}}{|\xi|^{\ell s_{2}}}d\xi&=&\int_{\varrho\geq R_{0}}\frac{e^{-ct/\varrho^{\sigma_{2}}}}{\varrho^{{\sigma_{2}}(\ell s_{2}/{\sigma_{2}})}}\varrho^{n-1}d\varrho\nonumber\\&=&\int^{\sqrt[{\sigma_{2}}]{t}/R_{0}}_{0}t^{-\frac{\ell s_{2}}{\sigma_{2}}}y^{\ell s_{2}}e^{-cy^{\sigma_{2}}}\frac{t^{\frac{n-1}{\sigma_{2}}}}{y^{n-1}}(t^{\frac{1}{\sigma_{2}}}y^{-2}dy)
 \nonumber\\&\lesssim& (1+t)^{-\frac{\ell s_{2}}{\sigma_{2}}+\frac{n}{\sigma_{2}}}
\ (\ell s_{2}>n). \label{R-E33}
\end{eqnarray}
Let us point out that the change of variables $\varrho=|\xi|$ and $y=\sqrt[\sigma_{2}]{t}/\varrho$ in the first and second lines of (\ref{R-E33})
were performed, respectively.

Together with (\ref{R-E32})-(\ref{R-E33}), we obtain
\begin{eqnarray}
I_{2}\lesssim (1+t)^{-\frac{\ell}{\sigma_{2}}+\frac{n}{\sigma_{2}}(\frac{1}{r}-\frac{1}{p})}\|\partial_{x}^{k+\ell}\phi\|_{L^r}, \label{R-E34}
\end{eqnarray}
where the constraint $\ell s_{2}>n$ leads to $\ell>n(\frac{1}{r}-\frac{1}{p})$. It should be noted that $I_{2}$ can be bounded by $(1+t)^{-\frac{\ell}{\sigma_{2}}}\|\partial_{x}^{k+\ell}\phi\|_{L^2}$ with $\ell\geq0$ if $p=r=2$.

Hence, combining (\ref{R-E30})-(\ref{R-E31}) and (\ref{R-E34}) together,
the proof of Theorem \ref{thm1.1} is complete immediately.
\end{proof}

\section{The proof of Theorem \ref{thm1.2}}\setcounter{equation}{0}\label{sec:3}
Based on the $L^p$-$L^q$-$L^r$ estimate for $(1,2)$-type in Corollary \ref{cor1.1}, the main objective of this section is to show the optimal decay estimate of $L^{1}$-$L^2$ type for (\ref{R-E10})-(\ref{R-E11}) under the minimal regularity assumption. For clarity, we separate the proof into two parts.

\subsection{Energy method in Fourier spaces}
Since earlier works \cite{SK,UKS}, the energy method in Fourier spaces have been well developed by the third author and his collaborators for
hyperbolic systems of viscoelasticity, hyperbolic-elliptic systems of radiating gas, compressible Euler-Maxwell equations, Timoshenko systems and the plate equation with rotational inertia effect
and so on, see \cite{DNK,HK,IK,SK2,UK} and therein references. The interested reader is also referred to \cite{UDK} for generally hyperbolic systems with non-symmetric dissipation. Usually, the energy method in Fourier spaces is adapted to linearized systems. Here,
we shall perform the nonlinear version in Fourier spaces for (\ref{R-E10})-(\ref{R-E11}), see (\ref{R-E18}) below.
Let us mention that the similar estimate was first given by the third author in \cite{Ka2} for the Boltzmann equation, then well developed in \cite{KY2} for hyperbolic systems of balance laws.

\begin{prop} \label{prop3.1}
 Let $z=(\rho, \upsilon, E, h)^{\top}$ be the global classical solutions constructed in \cite{UWK} (also see (\ref{R-E8})-(\ref{R-E9})). Then
the Fourier image of classical solutions of (\ref{R-E10})-(\ref{R-E11}) satisfies the following pointwise estimate
\begin{eqnarray}
|\hat{z}(\xi)|^2\lesssim e^{-c_{1}\eta(\xi)t}|\hat{z}_{0}(\xi)|^2+\int^{t}_{0}e^{-c_{1}\eta(\xi)(t-\tau)}(|\xi|^{2}|\hat{Q}(\tau,\xi)|^2+|\hat{R}(\tau,\xi)|^2)d\tau, \label{R-E18}
\end{eqnarray}
for any $t\geq0$ and $\xi\in \mathbb{R}^{3}$, where the dissipative rate $\eta(\xi):=|\xi|^2/(1+|\xi|^2)^2$ and $c_{1}>0$ is a constant.
\end{prop}

\begin{proof}
Indeed, it suffices to show the influence of nonlinear terms, since the proof follows from the energy method in Fourier spaces as in \cite{UK}.
Applying the Fourier transform to (\ref{R-E10}) gives
\begin{equation}
\left\{
\begin{array}{l}\partial_{t}\hat{\rho}+n_{\infty}i|\xi|\hat{\upsilon}\cdot\omega=0,\\
\partial_{t}\hat{\upsilon}+a_{\infty}i|\xi|\hat{\rho}\omega+\hat{E}+\hat{\upsilon}\times B_{\infty}+\hat{\upsilon}=(i|\xi|\hat{q}_{2}\cdot\omega+\hat{r}_{2})/n_{\infty},
\\
\partial_{t}\hat{E}+i|\xi|\hat{h}\times\omega-n_{\infty}\hat{\upsilon}=0,\\
\partial_{t}\hat{h}-i|\xi|\hat{E}\times\omega=0.\\
 \end{array} \right.\label{R-E19}
\end{equation}
Also, we have
\begin{eqnarray}
i|\xi|\hat{E}\cdot\omega=-\hat{\rho},\ \ \ i|\xi|\hat{h}\cdot\omega=0. \label{R-E20}
\end{eqnarray}
For clarity, we divide it into three steps.

\textit{Step 1.} (Estimate for dissipative term of $\hat{\upsilon}$)

Performing the inner product of (\ref{R-E19}) with $a_{\infty}\hat{\rho},\ n_{\infty}\hat{\upsilon}, \hat{E}$ and $\hat{h}$, respectively, then adding the resulting equalities together. We take the real part to get
\begin{eqnarray}
\frac{d}{dt}\mathcal{E}_{0}+c_{2}|\hat{\upsilon}|^2=\mathrm{Re}\langle i|\xi|\hat{q}_{2}\cdot\omega+\hat{r}_{2},\hat{\upsilon}\rangle, \label{R-E21}
\end{eqnarray}
where $\mathcal{E}_{0}:=a_{\infty}|\hat{\rho}|^2+n_{\infty}|\hat{\upsilon}|^2+|\hat{E}|^2+|\hat{h}|^2\approx |\hat{z}|^2$ and $c_{2}=2n_{\infty}$. Here and below, $\mathrm{Re}\hat{f}$
means the real part of $\hat{f}$. It follows from Young's inequality that
\begin{eqnarray}
\frac{d}{dt}\mathcal{E}_{0}+\frac{c_{2}}{2}\lesssim (|\xi|^2|\hat{Q}|^2+|\hat{R}|^2). \label{R-E22}
\end{eqnarray}

\textit{Step 2.} (Estimate for dissipative term of $(\hat{\rho},\hat{E})$)

Performing the inner product of the second and the third equations of (\ref{R-E19}) with $a_{\infty}i|\xi|\hat{\rho}\omega+\hat{E}$ and $\hat{\upsilon}$, respectively, and then adding the resulting equalities implies
\begin{eqnarray}
&&\{\langle a_{\infty}i|\xi|\hat{\rho}_{t}\omega,\hat{\upsilon}\rangle+\langle\hat{\upsilon}_{t},a_{\infty}i|\xi|\hat{\rho}\omega\rangle\}
+\{\langle\hat{\upsilon}_{t},\hat{E}\rangle+\langle\hat{E}_{t},\hat{\upsilon}\rangle\}
\nonumber\\&&\hspace{5mm}+|a_{\infty}i|\xi|\hat{\rho}\omega+\hat{E}|^2-n_{\infty}|\hat{\upsilon}|^2-n_{\infty}a_{\infty}|\xi|^2|\hat{\upsilon}\cdot\omega|^2
\nonumber\\&&\hspace{5mm} + \langle\hat{\upsilon}\times B_{\infty}+\hat{\upsilon},a_{\infty}i|\xi|\hat{\rho}\omega+\hat{E}\rangle+i\xi\langle\hat{h}\times\omega,\hat{\upsilon}\rangle
\nonumber\\&&\hspace{10mm} =\langle (i|\xi|\hat{q}_{2}\cdot\omega+\hat{r}_{2})/n_{\infty},a_{\infty}i|\xi|\hat{\rho}\omega+\hat{E}\rangle. \label{R-E23}
\end{eqnarray}
Taking the real part of (\ref{R-E23}) and just following from the similar procedure as in \cite{UK}, we arrive at
\begin{eqnarray}
&&\frac{d}{dt}(\mathcal{E}_{1}+a_{\infty}|\xi|\mathcal{E}_{2})+c_{3}(1+|\xi|^2)|\hat{\rho}|^2+c_{3}|\hat{E}|^2\nonumber\\&\leq& \epsilon \frac{|\xi|^2}{1+|\xi|^2}|\hat{h}|^2+C_{\epsilon}(1+|\xi|^2)|\hat{\upsilon}|^2+C(|\xi|^2|\hat{Q}|^2+|\hat{R}|^2), \label{R-E24}
\end{eqnarray}
for any $\epsilon>0$, where $\mathcal{E}_{1}:=\mathrm{Re}\langle\hat{\upsilon},\hat{E}\rangle$, $\mathcal{E}_{2}:=\mathrm{Re}\langle i\hat{\rho}\omega,\hat{\upsilon}\rangle$ and
$c_{3}, C_{\epsilon}$ (depending on $\epsilon$) are some positive constants.

\textit{Step 3.} (Estimate for dissipative term of $\hat{h}$)

Performing the inner product of the third and fourth equations of (\ref{R-E19}) with $i|\xi|\hat{h}\times\omega$
and $i|\xi|\hat{E}\times\omega$, respectively, adding the resulting equalities together, and then taking the real part gives
\begin{eqnarray}
\frac{d}{dt}(|\xi|\mathcal{E}_{3})+|\xi|^2|\hat{h}\times\omega|^2=|\xi|^2|\hat{E}\times\omega|^2-\mathrm{Re}\langle n_{\infty}\hat{\upsilon},i|\xi|\hat{h}\times\omega\rangle, \label{R-E25}
\end{eqnarray}
where $\mathcal{E}_{3}:=\mathrm{Re}\langle\hat{E}, i\hat{h}\times\omega\rangle$. Due to (\ref{R-E20}), we have $|\hat{h}\times\omega|\approx |\hat{h}|$.
Furthermore, it follows from Young's inequality that there exists $c_{4}>0$ such that
\begin{eqnarray}
\frac{d}{dt}(|\xi|\mathcal{E}_{3})+c_{4}|\xi|^2|\hat{h}|^2\lesssim |\xi|^2|\hat{E}|^2+C|\hat{\upsilon}|^2. \label{R-E26}
\end{eqnarray}

Together energy inequalities (\ref{R-E22}), (\ref{R-E24}) and (\ref{R-E26}), the next step is to make
the suitable linear combination for them. Here, we feel free to skip them,  see \cite{UK} for similar details. That is, the Euler-Maxwell system admits
Lyapunov function
$$\mathcal{E}[\hat{z}]:=\mathcal{E}_{0}+\frac{\alpha_{1}}{1+|\xi|^2}\Big\{\mathcal{E}_{1}+a_{\infty}|\xi|\mathcal{E}_{2}+\frac{\alpha_{2}|\xi|}{1+|\xi|^2}\mathcal{E}_{3}\Big\}$$
such that the following differential inequality holds
\begin{eqnarray}
\frac{d}{dt}\mathcal{E}[\hat{z}]+c_{1}\mathcal{D}[\hat{z}]\lesssim (|\xi|^2|\hat{Q}|^2+|\hat{R}|^2), \label{R-E27}
\end{eqnarray}
where $$\mathcal{D}[\hat{z}]=|\hat{\rho}|^2+|\upsilon|^2+\frac{1}{1+|\xi|^2}|\hat{E}|^2+\frac{|\xi|^2}{(1+|\xi|^2)^2}|\hat{h}|^2,$$
and $\alpha_{1},\alpha_{2}$ are suitable small constants which ensure that $\mathcal{E}[\hat{z}]\approx |\hat{z}|^2$.  It follows from (\ref{R-E27}) that
\begin{eqnarray}
\frac{d}{dt}\mathcal{E}[\hat{z}]+c_{1}\eta(\xi)\mathcal{E}[\hat{z}]\lesssim (|\xi|^2|\hat{Q}|^2+|\hat{R}|^2), \label{R-E28}
\end{eqnarray}
where $\eta(\xi)=|\xi|^2/(1+|\xi|^2)^2$. Finally, the inequality (\ref{R-E18}) is followed from Gronwall's inequality.
\end{proof}

\subsection{Optimal decay rate}
In what follows, with preparations of Corollary \ref{cor1.1} and Proposition \ref{prop3.1}, we proceed the optimal decay estimate for (\ref{R-E10}). To show the minimal decay regularity of classical solutions, we define new time-weighted energy functionals:
$$N(t)=\sup_{0\leq \tau\leq t}(1+\tau)^{\frac{3}{4}}\|z(\tau)\|_{L^2}, $$
$$D(t)^2=\int^{t}_{0}\Big(\|(\rho,\upsilon)(\tau)\|^2_{H^3}+\|E(\tau)\|^2_{H^2}+\|\nabla h(\tau)\|^2_{H^1}\Big)d\tau.$$
Furthermore,
the nice degenerate structure of nonlinear terms $Q(z)$ and $R(z)$ enables us to
deduce a nonlinear energy inequality, which is included in the following
\begin{prop}\label{prop3.2}
Let $z=(\rho, \upsilon, E, h)^{\top}$ be the global classical solutions of (\ref{R-E10})-(\ref{R-E11}) (see (\ref{R-E8})-(\ref{R-E9})). Additionally, if $z_{0}\in L^1$, then
\begin{eqnarray}
N(t)\lesssim \|z_{0}\|_{H^{3}\cap L^1}+N(t)D(t)+N(t)^2. \label{R-E35}
\end{eqnarray}
\end{prop}

\begin{proof} Let us begin with (\ref{R-E18}):
\begin{eqnarray}
\int_{\mathbb{R}^{3}}|\hat{z}(\xi)|^2d\xi &\lesssim& \int_{\mathbb{R}^{3}}e^{-c\eta(\xi)t}|\hat{z}_{0}(\xi)|^2d\xi\nonumber\\&&
+\int_{\mathbb{R}^{3}}\int^{t}_{0}e^{-c\eta(\xi)(t-\tau)}(|\xi|^{2}|\hat{Q}(\tau,\xi)|^2+|\hat{R}(\tau,\xi)|^2)d\tau d\xi
\nonumber\\&\triangleq& J_{1}+J_{2}+J_{3}. \label{R-E36}
\end{eqnarray}
For $J_{1}$, by taking $p=2, k=j=0, q=1$ and $r=\ell=2$ in Corollary \ref{cor1.1}, we arrive at
\begin{eqnarray}
J_{1}\lesssim (1+t)^{-\frac{3}{2}}\|z_{0}\|^2_{L^1}+(1+t)^{-2}\|\partial_{x}^2z_{0}\|^2_{L^2}. \label{R-E37}
\end{eqnarray}

Next, we begin to estimate nonlinear terms. For $J_{2}$, it is written as the sum of low-frequency and high-frequency
$$J_{2}:=J_{2L}+J_{2H}.$$
For $J_{2L}$, by taking $p=2$ and $k=1,j=0, q=1$ in Corollary \ref{cor1.1}, we have
\begin{eqnarray}
J_{2L}&\lesssim& \int^{t}_{0}(1+t-\tau)^{-\frac{5}{2}}\|Q(\tau)\|^2_{L^1}d\tau \nonumber\\&\lesssim&
\int^{t}_{0}(1+t-\tau)^{-\frac{5}{2}}\|z(\tau)\|^4_{L^2}d\tau\nonumber\\&\lesssim&
N(t)^4\int^{t}_{0}(1+t-\tau)^{-\frac{5}{2}}(1+\tau)^{-3}d\tau\nonumber\\&\lesssim&N(t)^4(1+t)^{-3}, \label{R-E38}
\end{eqnarray}
where we have used the fact $Q(z)=O(|(\rho,\upsilon)|^2)$. For simplicity, we set $z^{\bot}:=(\rho,\upsilon)$.

For $J_{2H}$, more elaborate estimates are needed. For this purpose, we write
$$J_{2H}=\Big(\int^{t/2}_{0}+\int^{t}_{t/2}\Big)(\cdot\cdot\cdot)d\tau:=J_{2H1}+J_{2H2}.$$
Taking $p=2$ and $k=1,\ell=2, r=2$ in Corollary \ref{cor1.1} gives
\begin{eqnarray}
J_{2H1}&\lesssim& \int^{t/2}_{0}(1+t-\tau)^{-2}\|\partial_{x}^{3}Q(\tau)\|^2_{L^2}d\tau
\nonumber\\ &\lesssim& \int^{t/2}_{0}(1+t-\tau)^{-2} \|z^{\bot}\|^2_{L^\infty}\|\partial^3_{x}z^{\bot}\|_{L^2}^2d\tau
\nonumber\\ &\lesssim& \sup_{0\leq \tau\leq t/2} \Big\{(1+t-\tau)^{-2} \|z\|^2_{L^\infty}\Big\} \int^{t/2}_{0}\|\partial^3_{x}z^{\bot}\|_{L^2}^2d\tau
\nonumber\\ &\lesssim& (1+t)^{-2}N_{0}^2(t)D^2(t)\nonumber\\ &\lesssim&
(1+t)^{-2}\|z_{0}\|^2_{H^3},  \label{R-E39}
\end{eqnarray}
where we have used (\ref{R-E9}) (taking $s=3$) and the fact $Q(z)=O(|z^{\bot}|^2)$.

On the other hand, by taking $p=2$ and $k=1,\ell=2, r=1$ in Corollary \ref{cor1.1}, we get \begin{eqnarray}
J_{2H2}&\lesssim& \int^{t}_{t/2}(1+t-\tau)^{-\frac{1}{2}}\|\partial_{x}^{3}Q(\tau)\|^2_{L^1}d\tau. \label{R-E40}
\end{eqnarray}
It follows from the Gagliardo-Nirenberg interpolation inequality in \cite{N} that
\begin{eqnarray}
\|\partial_{x}f\|_{L^2}\lesssim \|f\|_{L^2}^{2/3}\|\partial_{x}^{3}f\|_{L^2}^{1/3},\ \ \|\partial^2_{x}f\|_{L^2}\lesssim \|f\|_{L^2}^{1/3}\|\partial_{x}^{3}f\|_{L^2}^{2/3}. \label{R-E41}
\end{eqnarray}
Furthermore, together with (\ref{R-E40})-(\ref{R-E41}), we obtain
\begin{eqnarray}
J_{2H2}&\lesssim& \int^{t}_{t/2}(1+t-\tau)^{-\frac{1}{2}} \|z^{\bot}\|^2_{L^2}\|\partial_{x}^{3}z^{\bot}\|^2_{L^2}d\tau
\nonumber\\&\lesssim& N(t)^2\int^{t}_{t/2}(1+t-\tau)^{-\frac{1}{2}} (1+\tau)^{-\frac{3}{2}}\|\partial_{x}^{3}z^{\bot}\|^2_{L^2}d\tau
\nonumber\\&\lesssim& N(t)^2 \sup_{t/2\leq\tau \leq t}\Big\{(1+t-\tau)^{-\frac{1}{2}} (1+\tau)^{-\frac{3}{2}}\Big\}\int^{t}_{0}\|\partial_{x}^{3}z^{\bot}\|^2_{L^2}d\tau
\nonumber\\&\lesssim& (1+t)^{-\frac{3}{2}} N(t)^2 D(t)^2. \label{R-E42}
\end{eqnarray}

For $J_{3}$, we write $$J_{3}:=J_{3L}+J_{3H}.$$
Note that $R(z)=O(\rho |E|+|\upsilon||h|)$, by taking $p=2$ and $k=j=0, q=1$ in Corollary \ref{cor1.1}, we obtain
\begin{eqnarray}
J_{3L}&\lesssim& \int^{t}_{0}(1+t-\tau)^{-\frac{3}{2}}\|R(\tau)\|^2_{L^1}d\tau \nonumber\\&\lesssim&
\int^{t}_{0}(1+t-\tau)^{-\frac{3}{2}}\|z(\tau)\|^4_{L^2}d\tau\nonumber\\&\lesssim&
N(t)^4\int^{t}_{0}(1+t-\tau)^{-\frac{3}{2}}(1+\tau)^{-3}d\tau\nonumber\\&\lesssim&N(t)^4(1+t)^{-3}. \label{R-E43}
\end{eqnarray}
Similarly, we separate the high-frequency part $J_{3H}$ as follows
$$J_{3H}=\Big(\int^{t/2}_{0}+\int^{t}_{t/2}\Big)(\cdot\cdot\cdot)d\tau:=J_{3H1}+J_{3H2}.$$
Taking $p=2$ and $k=0, \ell=r=2$ in Corollary \ref{cor1.1} leads to
\begin{eqnarray}
J_{3H1}&\lesssim &\int^{t/2}_{0}(1+t-\tau)^{-2}\|\partial_{x}^2R(\tau)\|^2_{L^2}d\tau
\nonumber\\ &\lesssim& \int^{t/2}_{0}(1+t-\tau)^{-2} \|z\|^2_{L^\infty}\|\partial_{x}^2z\|_{L^2}^2d\tau
\nonumber\\ &\lesssim& \sup_{0\leq \tau\leq t/2}\Big\{(1+t-\tau)^{-2}\|z\|^2_{L^\infty}\Big\}\int^{t/2}_{0}\|\partial_{x}^2z\|_{L^2}^2d\tau
\nonumber\\ &\lesssim& (1+t)^{-2}N_{0}^2(t)D^2(t)\nonumber\\ &\lesssim&
(1+t)^{-2}\|z_{0}\|^2_{H^3},  \label{R-E44}
\end{eqnarray}
where we have used (\ref{R-E9}) (taking $s=3$).
On the other hand, by taking $p=2$ and $k=0, \ell=2, r=1$ in Corollary \ref{cor1.1}, we arrive at
\begin{eqnarray}
J_{3H2}&\lesssim& \int^{t}_{t/2}(1+t-\tau)^{-\frac{1}{2}}\|\partial_{x}^2R(\tau)\|^2_{L^1}d\tau
\nonumber\\&\lesssim& \int^{t}_{t/2}(1+t-\tau)^{-\frac{1}{2}}\|z\|^2_{L^2}\|\partial_{x}^2z\|^2_{L^2}
d\tau
\nonumber\\&\lesssim& N(t)^2\int^{t}_{t/2} (1+t-\tau)^{-\frac{1}{2}} (1+\tau)^{-\frac{3}{2}}\|\partial_{x}^2z\|^2_{L^2}
d\tau
\nonumber\\&\lesssim& N(t)^2 \sup_{t/2\leq\tau \leq t}\Big\{(1+t-\tau)^{-\frac{1}{2}} (1+\tau)^{-\frac{3}{2}}\Big\}\int^{t}_{0}\|\partial_{x}^2z\|^2_{L^2}d\tau
\nonumber\\&\lesssim& (1+t)^{-\frac{3}{2}} N(t)^2 D(t)^2, \label{R-E45}
\end{eqnarray}
where the Gagliardo-Nirenberg inequality $\|\partial_{x}f\|_{L^2}\lesssim \|f\|^{1/2}_{L^2}\|\partial^{2}_{x}f\|^{1/2}_{L^2}$
was used in the second line.

Therefore, combining above inequalities (\ref{R-E37})-(\ref{R-E39}) and (\ref{R-E42})-(\ref{R-E45}), it follows from Plancherel's theorem that
\begin{eqnarray}
\|z\|^2_{L^2}\lesssim (1+t)^{-\frac{3}{2}}\|z_{0}\|^2_{H^{3}\cap L^1}+(1+t)^{-3}N(t)^4+(1+t)^{-\frac{3}{2}} N(t)^2 D(t)^2 \label{R-E46}
\end{eqnarray}
which leads to
\begin{eqnarray}
N(t)\lesssim \|z_{0}\|_{H^{3}\cap L^1}+N(t)D(t)+N(t)^2, \label{R-E47}
\end{eqnarray}
which is (\ref{R-E35}) exactly.
\end{proof}

According to the energy inequality (\ref{R-E9}) (taking $s=3$), the dissipation norm $D(t)\lesssim \|z_{0}\|_{H^{3}}\lesssim \|z_{0}\|_{H^{3}\cap L^1}$. Thus,
if $\|z_{0}\|_{H^{3}\cap L^1}$ is sufficient small, then it holds that
\begin{eqnarray}
N(t)\lesssim \|z_{0}\|_{H^{3}\cap L^1}+N(t)^2 \label{R-E48}
\end{eqnarray}
which implies that $N(t)\lesssim \|z_{0}\|_{H^{3}\cap L^1}$, provided that $\|z_{0}\|_{H^{3}\cap L^1}$ is sufficient small.
Consequently, the optimal decay estimate in Theorem \ref{thm1.2} is achieved.

\section{Appendix}\setcounter{equation}{0} \label{sec:4}
In the last section, as another application of Corollary \ref{cor1.1}, we generalize recent decay properties in  \cite{UDK}
for linear symmetric hyperbolic systems with non-symmetric dissipation.

\subsection{Symmetric hyperbolic systems}
Consider the Cauchy problem for the first-order linearized symmetric hyperbolic
system of equations with  dissipation
\begin{equation}
\left\{
\begin{array}{l}
A^{0}w_{t}+\sum^{n}_{j=1}A^{j}w_{x_{j}}+Lw=0, \\
w|_{t=0}=w_{0},
\end{array} \right.\label{R-E49}
\end{equation}
with $w(t,x)\in \mathbb{R}^{m}$ for $t>0$ and $x\in\mathbb{R}^{n}$, where
$A^{j}(j=0,1,\cdot\cdot\cdot,n)$ and $L$ are $m\times m$ real constant
matrices. It is assumed that all $A^{j}(j=0,1,\cdot\cdot\cdot,n)$ are symmetric, $A_{0}$ is positive
definite and $L$ is nonnegative definite with a nontrivial kernel.

If the degenerate dissipation matrix $L$ is symmetric, the third author and Shizuta \cite{SK} first formulated the so-called Kawashima-Shizuta condition which designs the compensating matrix $K$ to capture the dissipation of systems over the degenerate kernel space of $L$. Inspired by recent concrete examples, such as dissipative
Timoshenko systems and compressible Euler-Maxwell equations and so on, the matrix $L$ has the skew-symmetric part and is not symmetric. In this case,
the partial positivity on $\mathrm{Ker}(L_{1})^{\bot}$ ($L_{1}:=$ the symmetric part of $L$) is available only.
Recently, Ueda, Duan and the third author \cite{UDK} found a real compensating matrix $S$ to make up the full positivity on $\mathrm{Ker}(L)^{\bot}$. Consequently, they developed decay properties for (\ref{R-E49}) with the weaker dissipative mechanism. Here, we don't collect those structural conditions formulated by \cite{UDK} for brevity, however, we would like to keep the same notations as in \cite{UDK} for the convenience of reader.

Based on \textbf{conditions} (A), (K), (S) and (S)$_{1}$ in \cite{UDK}, by employing the energy method in Fourier spaces, they arrived at
\begin{eqnarray}
\frac{d}{dt}\mathcal{E}(t,\xi)+c\eta(\xi)\mathcal{E}(t,\xi)\leq0, \label{R-E50}
\end{eqnarray}
which implies that $\mathcal{E}(t,\xi)\leq e^{-c\eta(\xi)t}\mathcal{E}(0,\xi)$, where $\eta(\xi)=|\xi|^2/(1+|\xi|^2)^2$ and $\mathcal{E}(t,\xi)\approx |\hat{w}|^2$.
Furthermore, from Corollary \ref{cor1.1}, we can generalize the decay property in \cite{UDK}.
\begin{prop}\label{prop4.1}
Assume that \textbf{conditions} (A), (K), (S) and (S)$_{1}$ in \cite{UDK} hold. If
the initial data $w_{0}\in W^{l,r}\cap L^{q}$ for $l\geq0$ and $1\leq q,r \leq2$,
then the solution $w(t,x)$ of (\ref{R-E49}) satisfies the decay estimate
\begin{equation}
\|\partial_{x}^{k}w\|_{L^p}\lesssim (1+t)^{-\gamma_{2}(q,p)-\frac{k}{2}}\|w_{0}\|_{L^{q}}
+(1+t)^{-\frac{\ell}{2}+\gamma_{2}(r,p)} \|\partial^{k+\ell}_{x}w_{0}\|_{L^{r}}\label{R-E51}
\end{equation}
for $\ell>n(\frac{1}{r}-\frac{1}{p})$\ \footnote{ Here $\ell\geq0$ when $p=r=2$.},\ $2\leq p\leq\infty$ and $0\leq k+\ell\leq l$.
\end{prop}

\subsection{Symmetric hyperbolic systems with constraints}
Inspired by Euler-Maxwell equations, the system (\ref{R-E49}) equipped with a general constraint
was also investigated in \cite{UDK}:
\begin{equation}
\sum_{j=1}^{n}\mathcal{Q}^{j}w_{x_{j}}+\mathcal{R}w=0, \label{R-E52}
\end{equation}
where $\mathcal{Q}^{j}$ and $\mathcal{R}$ are $m_{1}\times m$ real constant matrices with $m_{1}<m$.
Let $\Pi_{1}$ be the orthogonal projection from $\mathbb{C}^{m_{1}}$ onto Image$(\mathcal{R})=\{\mathcal{R}\phi:\phi\in \mathbb{C}^{m}\}\subset\mathbb{C}^{m_{1}}$. Set $\Pi_{2}=I-\Pi_{1}$. Noticing that $\Pi_{1}$ and $\Pi_{2}$ are $m_{1}\times m_{1}$ real symmetric matrices. Using these projections, the condition (\ref{R-E52})
can be decomposed as
\begin{eqnarray}
\sum_{j=1}^{n}\Pi_{1}\mathcal{Q}^{j}w_{x_{j}}+\mathcal{R}w=0,\ \ \ \ \sum_{j=1}^{n}\Pi_{2}\mathcal{Q}^{j}w_{x_{j}}=0. \label{R-E53}
\end{eqnarray}
To ensure that (\ref{R-E52}) or (\ref{R-E53}) holds at an arbitrary time $t>0$ if it satisfies initially, the extra structure \textbf{condition} (C) is posted, see \cite{UDK} for details. Additionally, those conditions in Proposition \ref{prop4.1} need to be revised a little in this case.  For convenience, the same notations as in \cite{UDK} are kept. As a consequence, the dissipative inequality (\ref{R-E50}) still holds for the solution of (\ref{R-E49}) along with (\ref{R-E52}). Furthermore, we have a similar decay property as stated in Proposition \ref{prop4.1}.
\begin{prop}\label{prop4.2}
Assume that \textbf{conditions} (A), (C), (S),$(S^{*})_{1}$ and(K$^{*})$ in \cite{UDK} hold. If
the initial data $w_{0}\in W^{l,r}\cap L^{q}$ for $l\geq0$ and $1\leq q,r \leq2$,
then the solution $w(t,x)$ of (\ref{R-E49}) satisfies (\ref{R-E52}) for all $t>0$. Moreover,
the solution satisfies the decay estimate (\ref{R-E51}).
\end{prop}

\begin{rem}\label{rem4.1}
Propositions \ref{prop4.1}-\ref{prop4.2} go back to Theorem 2.2 and Theorem 5.2 in \cite{UDK},
if one takes $p=r=2$ and $q=1$. Therefore, the current decay properties can be regarded as a general $L^p$-version.
\end{rem}

\section*{Acknowledgments}
J. Xu is partially supported by the National
Natural Science Foundation of China (11471158), the Program for New Century Excellent
Talents in University (NCET-13-0857) and the NUAA Fundamental
Research Funds (NS2013076). He would like to thank Professor Kawashima
for giving him much help when he was visiting Kyushu University in Japan.
The work is also partially supported by
Grant-in-Aid for Scientific Researches (S) 25220702 and (A) 22244009.

\end{document}